\documentclass[12pt]{amsart}

\usepackage[margin = 1.3in]{geometry}

\usepackage{tikz}
\usepackage{tikz-cd}
\usepackage{url, hyperref}
\usepackage{float}

\tikzset{->-/.style={decoration={
  markings,
  mark=at position .45 with {\arrow{>}}},postaction={decorate}}}
  
\usetikzlibrary{shapes.arrows}
\usetikzlibrary{decorations.pathreplacing}
\usepackage[all]{xy}

\tikzset{->-/.style={decoration={
  markings,
  mark=at position .45 with {\arrow{>}}},postaction={decorate}}}

\usepackage{amsmath}
\usepackage{amssymb}
\usepackage{enumitem}
\usepackage{graphicx}
\usepackage{mathdots}
\usepackage{color}
\usepackage{diagbox}
\usepackage{array, makecell}
\usepackage{rotating}
\usepackage{amsmath}
\usepackage{tikz-cd}

\def\cG{\mathcal{G}}

\def\cM{{\mathcal{M}}}

\def\Z{\mathbb{Z}}

\def\b1{{\bf 1}}

\def\E{\mathsf{E}}

\def\H{\mathsf{H}}

\def\P{\mathbb{P}}
\def\bP{\mathbb{P}}

\def\cP{\mathcal{P}}

\def\cL{\mathcal{L}}

\newcount\colveccount
\newcommand*\colvec[1]{
        \global\colveccount#1
        \begin{pmatrix}
        \colvecnext
}
\def\colvecnext#1{
        #1
        \global\advance\colveccount-1
        \ifnum\colveccount>0
                \\
                \expandafter\colvecnext
        \else
                \end{pmatrix}
        \fi
}

\usepackage{amsthm}
\newtheorem{definition}{Definition}[section]
\newtheorem{theorem}[definition]{Theorem}
\newtheorem{example}[definition]{Example}
\newtheorem{proposition}[definition]{Proposition}
\newtheorem{corollary}[definition]{Corollary}
\newtheorem{lemma}[definition]{Lemma}
\newtheorem{remark}[definition]{Remark}

\newtheorem{question}[definition]{Question}

\newcommand{\Pic}{\mathsf{Pic}}

\newcommand{\Bl}{{\mathsf{Bl}}}

\newcommand{\ev}{{\text{ev}}}

\usepackage{amsmath,calligra,mathrsfs}
\DeclareMathOperator{\cHom}{\mathscr{H}\text{\kern -3pt {\calligra\large om}}\,}

\def\cL{\mathcal{L}}

\title{A Brill-Noether theorem for (toric) surfaces}

\author{Alessio Cela}
\address{ University of Cambridge, Department of pure mathematics and mathematical statistics
\hfill \newline\texttt{}
 \indent Centre for Mathematical Sciences, Wilberforce Road Cambridge, UK} \email{{\tt ac2758@cam.ac.uk}}

 \author{Carl Lian}
\address{Washington University in St. Louis, Department of Mathematics, 1 Brookings Drive
\hfill \newline\texttt{}
 \indent  St. Louis, MO 63130} \email{{\tt clian@wustl.edu}}

\date{\today}

\usepackage{graphicx}

\begin{document}

\maketitle

\begin{abstract}
The classical Brill-Noether theorem states that a map from a general curve to a projective space deforms in a family of expected dimension as long as its image does not lie in any hyperplane. In this note, we observe, as a direct consequence of standard results on Severi varieties, an analogous statement for maps from a general curve to any smooth, projective surface. Namely, a non-constant map deforms in a family of expected dimension as long as its image has anti-canonical degree at least 4. 

In the case of toric surfaces, curves of anti-canonical degree at most 3 admit a particularly elegant description in terms certain toric contractions. We raise the question of whether a Brill-Noether theorem could hold for toric varieties of higher dimension.
% The classical Brill-Noether theorem implies \Ale{is equivalent to} that a map from a general curve to a projective space deforms in a family of expected dimension as long as its image does not lie in any hyperplane. In this paper, we conjecture a generalization for maps to toric varieties: a map deforms in a family of expected dimension as long as its image does not lie in any hyperplane after post-composing with a toric morphism to a fake projective space. The conjecture is proven in dimension 2, and is also consistent with a result of Farkas for blow-ups of projective spaces. We also discuss applications to enumerative geometry.
\end{abstract}

\section{Introduction}

Let $C$ be a general curve over $\mathbb{C}$, by which we mean a general point of $\cM_g$. Classical Brill-Noether theory concerns maps from $C$ to projective spaces. One formulation of the Brill-Noether theorem is the following.

\begin{theorem}\label{thm:BN_classical}
    Suppose that $f:C\to\P^r$ is a morphism of degree $d$ which is \emph{non-degenerate}, that is, the image of $f$ does not lie in any hyperplane. Then, the map $f$ moves in a family of expected dimension. That is, the moduli space of maps $\cM_d(C,\P^r)$ of maps $C \to \P^r$ of degree $d$ is pure of dimension
    \begin{equation*}
        (r+1)d+r(1-g),
    \end{equation*}
    the expected, in a neighborhood of $f$.
\end{theorem}

 In fact, the Gieseker-Petri theorem asserts further that the moduli space of maps is \emph{smooth} at $f$. Indeed, a map $f$ whose image does not lie in a hyperplane is given, up to automorphisms on $\P^r$, by a base-point free, rank $r$ linear series, that is, a line bundle $\cL$ on $C$ and a $(r+1)$-dimensional space of sections $V\subset H^0(C,\cL)$. The moduli space $\cG^r_d(C)$ of rank $r$, degree $d$ linear series is smooth of dimension $\rho(g,r,d)=g-(r+1)(g-d+r)$ \cite[Theorem V.1.6]{acgh}.

In this note, we prove the following analogous statement for maps to a smooth, projective surface.

\begin{theorem}\label{thm:BN_surface}
Let $S$ be a smooth, projective surface. Let $f:C\to S$ be a morphism with $f_{*}[C]=\beta\neq0$, and let $C'=f(C)$ be its scheme-theoretic image. Suppose that $[C']\cdot (-K_S)\ge 4$. Then, the map $f$ moves in a family of expected dimension. That is, the moduli space of maps $\cM_\beta(C,S)$ is pure of dimension
    \begin{equation*}
        \beta\cdot (-K_S)+2(1-g),
    \end{equation*}
    the expected, in a neighborhood of $f$. The moduli space $\cM_\beta(C,S)$ is furthermore generically smooth in a neighborhood of $f$.
\end{theorem}

When $S=\P^2$, the curves of anti-canonical degree at most 3 are exactly lines, so Theorem \ref{thm:BN_surface} recovers the classical Brill-Noether theorem when $r=2$. Theorem \ref{thm:BN_surface} is proven in \S\ref{sec:severivarieties}. As we will see, it is a consequence of well-known transversality results for Severi varieties of \emph{embedded} curves on $S$; the work is to compare dimensions for moduli of maps which may or may not be birational onto their image. This comparison is also standard, see for example \cite[\S 3]{vakil}. 

If $S$ is a \emph{toric} surface, then the loci on $\cM_\beta(C,S)$ that must be removed in order to obtain a moduli space that is pure of expected dimension, namely the loci of multiple covers of curves of anti-canonical degree at most 3, admit a particularly simple description. The following is proven in \S\ref{sec:toric_surface}.

\begin{theorem}\label{thm:toric_low_degree}
Let $S$ be a smooth, projective toric surface. Let $C'\subset S$ be an integral curve with $[C']\cdot (-K_S)\le 3$, and assume that $C'$ is not a component of the toric boundary. Then, one of the following holds:
\begin{enumerate}
    \item[(i)] $[C']\cdot (-K_S)=2$, in which case there exists a toric morphism $\pi:S\to \P^1$ contracting $C'$ to a point, or
    \item[(ii)] $[C']\cdot (-K_S)=3$, in which case there exists a \emph{fake toric projective plane} $S'$ (Definition \ref{def:fake_P2}) and a birational toric morphism $\pi:S\to S'$ induced by deletion of rays of the fan of $S$, such that $\pi(C')$ lies in the primitive linear system of $S'$.
\end{enumerate}
\end{theorem}

We suspect that Theorems \ref{thm:BN_surface} and \ref{thm:toric_low_degree} are known in some form to experts, but we were not able to locate these precise statements in the literature. Combining these statements, one obtains a ``Brill-Noether theorem'' for toric surfaces.

\begin{corollary}\label{cor:BN_toric_surface}
Let $S$ be a smooth, projective toric surface. Let $f:C\to S$ be a morphism with $f_{*}[C]=\beta\neq0$, and let $C'=f(C)$ be its scheme-theoretic image. Assume that $C'$ is not contained in the toric boundary of $S$, and furthermore that $C'$ satisfies neither (i) nor (ii) of Theorem \ref{thm:toric_low_degree}. Then, the map $f$ moves in a family of expected dimension. The moduli space $\cM_\beta(C,S)$ is furthermore generically smooth in a neighborhood of $f$.
\end{corollary}

In other words, if $f$ fails to move in a family of expected dimension, then the failure is explained by the degeneracy of $f$ after post-composing with a certain toric morphism to either a projective line or to a fake toric projective plane.

\begin{question}\label{question:higher_dim}
    Does Corollary \ref{cor:BN_toric_surface} admit a generalization to smooth, projective toric varieties of any dimension?
\end{question}

Classical Brill-Noether theory concerns sections of line bundles on curves, and maps to toric varieties are also governed by sections of line bundles on the source \cite{cox}, so this seems to be a natural generalized setting for Brill-Noether theory. While the methods of this note are highly specific to the case of surfaces, some hope that an extension to higher dimension could be possible is afforded by the following result of Farkas.

\begin{theorem}\cite[Theorem 0.1]{farkas}\label{thm:farkas}
    Let $b:Y\to \P^r$ be the blow-up along a (torus-invariant) linear space. Let $f:C\to Y$ be a morphism with $f_{*}[C]=\beta$. Suppose that $b(f(C))\subset\bP^r$ is not contained in any hyperplane. Then, the moduli space of maps $\cM_\beta(C,Y)$ is pure of dimension
    \begin{equation*}
        \beta\cdot(-K_Y)+r(1-g),
    \end{equation*}
    the expected, in a neighborhood of $f$.
\end{theorem}

 Farkas's result is stated in the language of linear series with secants, but the formulation above is equivalent, see \cite[Corollary 2.3.4]{cl_complete}. See also \cite{cotterill, ungureanu, farkas2} for closely related results.

 Recall that the Gieseker-Petri theorem asserts further that the moduli space of non-degenerate maps from a general curve to $\mathbb{P}^r$ is smooth. Equivalently, if $f:C\to\bP^r$ is non-degenerate, then it is unobstructed, that is, we have $H^1(C,f^{*}T_{\bP^r})=0$. We do not know whether this remains the case in the settings of Corollary \ref{cor:BN_toric_surface} and Theorem \ref{thm:farkas}.

 \begin{question}\label{question:smoothness}
     Let $f:C\to S$ (resp. $f:C\to Y$) be a map as in Corollary \ref{cor:BN_toric_surface} (resp. Theorem \ref{thm:farkas}). Is the moduli space $\cM_\beta(C,S)$ (resp. $\cM_\beta(C,Y)$) smooth at $f$?
 \end{question}

 Whereas Corollary \ref{cor:BN_toric_surface} gives that $\cM_\beta(C,S)$ is \emph{generically} smooth in a neighborhood of $f$, generic smoothness of $\cM_\beta(C,Y)$ in the setting of Theorem \ref{thm:farkas} does not seem to follow directly from Farkas's proof. Note that Questions \ref{question:higher_dim} and \ref{question:smoothness} are only interesting when $C$ has positive genus. Indeed, a rational curve generically contained in the interior of a smooth, projective toric variety is free, and therefore unobstructed. (Rational curves contained in the boundary are more naturally viewed as maps to the appropriate toric stratum.)

\subsection{Acknowledgments}

We thank Renzo Cavalieri, Gavril Farkas, François Greer, Xuanchun Lu, Navid Nabijou, Dhruv Ranganathan, Terry (Dekun) Song, Montserrat Teixidor i Bigas, and Sameera Vemulapalli for conversations related to these ideas, and the referee for improvements to the exposition. Some work on this note took place at the NSF-funded ``Workshop on Tevelev degrees and related topics'' at the University of Illinois, Urbana-Champaign in October 2024; we thank Deniz Genlik and Felix Janda for their work in organizing this pleasant event. A.C. is supported by SNF grant P500PT-222363. C.L. has been supported by NSF Postdoctoral Fellowship DMS-2001976 and an AMS-Simons travel grant.

\section{Severi varieties and moduli of maps}\label{sec:severivarieties}

 Let $S$ be a smooth projective surface, and let $\beta\in H_2(S,\Z)$ be a non-zero curve class on $S$. 
 
 \begin{definition}
     Let $Y_{\beta,g}(S)$ be the Severi variety of reduced, irreducible embedded curves $C'\subset S$ of geometric genus $g$ with $[C']=\beta$. 
 \end{definition}
 
 The following facts are standard; see, for example Harris-Morrison \cite[Corollary 3.42]{book-HM} and its proof. The results are stated there for $S=\P^2$, but only the positivity assumptions on the curve class $\beta$ are needed.

\begin{theorem}\label{severi_thm}
    We have the following.
    \begin{enumerate}
     \item[(i)] If $\beta\cdot (-K_S)\ge 1$, then every component of $Y_{\beta,g}(S)$ has the expected dimension of 
     \begin{equation*}
         \beta\cdot (-K_S)+g-1.
     \end{equation*}
     \item[(ii)] If $\beta\cdot (-K_S)\ge 4$, then any general point of $Y_{\beta,g}(S)$ corresponds to a nodal curve $C\subset S$.
     \item[(iii)] If $C'\subset S$ is nodal and $\beta\cdot (-K_S)\ge 1$, then $Y_{\beta,g}(S)$ is smooth of expected dimension at $[C']$.
     \end{enumerate}
\end{theorem}

We work instead with moduli spaces of maps.

\begin{definition}
    Let $C$ be a fixed smooth, projective, irreducible curve of genus $g$, and let $\beta\in H_2(S,\Z)$ be a non-zero curve class on $S$. Let $\cM_\beta(C,S)$ be the moduli space of maps $f:C\to S$ with $f_{*}[C]=\beta$.
\end{definition}

We say that such an $f$ is ``in class $\beta$.'' Two maps are only considered isomorphic if they are \emph{equal}, that is, an isomorphism of maps must restrict to the identity on both $C$ and $S$. The expected dimension of $\cM_\beta(C,S)$ is equal to
\begin{equation*}
    \chi(C,f^{*}T_S)=\beta\cdot (-K_S)-2g+2,
\end{equation*}
which is $\dim(\cM_g)=3g-3$ less than that of $Y_{\beta,g}(S)$.

There is a map $\mu:Y_{\beta,g}(S)\to \cM_g$ remembering the normalization $C$ of an embedded curve $C'\subset S$. Near a point $f\in \cM_\beta(C,S)$ birational onto its image, the space $\cM_\beta(C,S)$ is isomorphic to the fiber of $\mu$ over the geometric point $[C]\in \cM_g$. When $g=0$, this may be interpreted as the statement that there is a rational $\text{PGL}_2$-bundle $\nu:\cM_\beta(\P^1,S) \dashrightarrow Y_{\beta,0}(S)$ sending a map $f:\P^1\to S$ to $f(\P^1)$ whenever it is birational onto its image. Similarly, when $C$ has genus $1$, the non-empty fibers of the map $\nu:\cM_\beta(C,S) \dashrightarrow Y_{\beta,1}(S)$ remembering the image of a birational map are isomorphic to the 1-dimensional group scheme $\text{Aut}(C)$. 

 In particular, by Theorem \ref{severi_thm}, if $\beta\cdot (-K_S)\ge 4$ and $C$ is general, then $\cM_\beta(C,S)$ is generically smooth of the expected dimension near any $f$ birational onto its image. However, this may not be the case everywhere on the moduli space of maps, due to presence of multiple covers, where $f$ is not birational onto its image. Controlling these multiple covers is the main task in what follows.

 \begin{proposition}\label{prop:multiple covers}
 Let $C$ be a general curve and let $S$ be a smooth, projective surface. Let $\beta\in H_2(S,\mathbb{Z})$ be non-zero, effective curve class. Let $\cM\subset \cM_\beta(C,S)$ be an irreducible component.

 Let $f:C\to S$ be a general point of $\cM$, and let $C'\subset S$ be the scheme-theoretic image of $f$. Suppose that $[C']\cdot (-K_S) \ge 4$. Then, $\cM$ is generically smooth of expected dimension.
 \end{proposition}

 \begin{proof}
Let $m$ be the degree of induced cover $g:C\to C'\subset S$. If $m=1$, then we are done by the above discussion, as $\cM_{\beta}(C,S)$ is isomorphic near $f$ to a general fiber of $\mu:Y_{\beta,g}(S)\to\cM_g$. We assume henceforth that $m>1$. In this case, the general curve $C$ admits no covers of a curve of positive geometric genus of degree $m>1$, unless $g=g(C')=1$. We consider the cases $g(C')=0$ and $g=g(C')=1$ separately.

%     \begin{enumerate}[label=\underline{Case \arabic*}]
%    \item  
Suppose first that $g(C')=0$. The moduli space of maps $f:C\to \P^1$ of degree $m$ has dimension $\rho(g,1,m)+\dim(\text{PGL}_2)=2m-g+1$. In particular, we need $\rho(g,1,m)=2m-2-g\ge0$.
        % \begin{equation}\label{eqn: estimate 1 in exc dim}
        % g\le 2m-2.
        % \end{equation}
By Theorem \ref{severi_thm}, the moduli space of embedded curves $C'\subset S$ in class $\beta':=[C']$ has dimension $\beta'\cdot (-K_S)-1$. Therefore, the component $\cM$ has dimension
        \begin{equation*}
            (2m-g+1)+(\beta'\cdot (-K_S)-1) \ge \beta\cdot (-K_S)-2g+2,
        \end{equation*}
        where the right hand side is the expected dimension. Rearranging and substituting $\beta=m\beta'$ gives
       \begin{equation*}
              (m-1)(\beta'\cdot (-K_S)-2) \le g.
         \end{equation*} 
         
Recall that we also have $g\le 2m-2$. Thus, the given multiple cover $f$ can only give rise to a component $\cM$ if $\beta'\cdot (-K_S)=4$ and $g=2m-2$, hence $\rho(g,1,m)=0$. In this case, the space $\cM_\beta(C,S)$ is isomorphic near $f$ to the 6-dimensional space of maps which factor as 
        \begin{equation*}
            C\to \P^1\to X,
        \end{equation*}
        where $f:C\to\P^1$ corresponds to one of a \emph{reduced} (for example, by the Gieseker-Petri theorem) set of points parametrizing rank 1 linear series on $C$ of minimal degree $m=\frac{g+2}{2}$, and the image of $\P^1\to S$ is a point of the Severi variety $Y_{\beta',0}(S)$. By Theorem \ref{severi_thm}, the Severi variety is generically smooth 
        of expected dimension, so $\cM$ is as well.
        
Consider now the exceptional case $g(C')=g(C)=1$. Since $C$ admits only finitely many maps of degree $m>1$ to some genus $1$ curve, by Theorem \ref{severi_thm}, the component $\cM$ would have dimension $\beta' \cdot (-K_S)$, which is less than the expected dimension $\beta\cdot (-K_S)$, a contradiction.

%\end{enumerate}
 \end{proof}

 \begin{proof}[Proof of Theorem \ref{thm:BN_surface}]
% Immediate from Proposition \ref{prop:multiple covers}.
Let $\cM$ be any irreducible component of $\cM_\beta(C,S)$ containing $f$. Then, $\cM$ satisfies the hypothesis of Proposition \ref{prop:multiple covers}, because the condition on anti-canonical degree of $C'$ is open. Theorem \ref{thm:BN_surface} follows.
 \end{proof}

\begin{remark}
    The proof of Proposition \ref{prop:multiple covers} also shows that, if the Severi variety $Y_{\beta,g}(S)$ is irreducible, then the open locus of $\cM_\beta(C,S)$ consisting of maps whose image has anti-canonical degree at least 5 is also irreducible. Severi varieties are known in many cases to be irreducible, see for example \cite{harris_severi,testa,cd_severi,cht1,cht2}.
\end{remark}

\section{Toric surfaces}\label{sec:toric_surface}

We now specialize to the case of a smooth, projective toric surface $S$. By Proposition \ref{prop:multiple covers}, for a general curve $C$ and any non-zero curve class $\beta$, any component of the moduli space of maps $\cM_{\beta}(C,S)$ is generically smooth of expected dimension as long as its general point is not a multiple cover of a curve of anti-canonical degree at most 3. When $S=\P^2$, these curves are precisely the lines, so we recover the non-degeneracy condition in the classical Brill-Noether theorem. In this section, we prove Theorem \ref{thm:toric_low_degree}, which classifies the (non-boundary) curves of low anti-canonical degree on any smooth, projective toric surface $S$.

\begin{definition}\label{def:fake_P2}
Let $N\cong\Z^2$ be a lattice. A \emph{fake toric projective plane} $S'$ is a projective toric surface associated to a complete fan $N$ whose distinct 1-dimensional rays $n_1,n_2,n_3\in N$ satisfy $n_1+n_2+n_3=0$.
\end{definition}

We abuse terminology slightly, identifying a ray in $N$ with a primitive vector that spans it. If any two of $n_1,n_2,n_3$ form an integral basis for $N$, then $S'\cong\mathbb{P}^2$; otherwise, the surface $S'$ is singular. See Example \ref{example: fake plane} for a singular example. See also \cite{buzynska,conrads,kasprzyk,cgkn,bauerle} for work on fake toric \emph{weighted} projective spaces of arbitrary dimension.

Fake toric projective planes necessarily have $\Pic(S')\cong\mathbb{Z}$ \cite[\S 3.4]{fulton_toric}, so in particular, a natural notion of degeneracy of a non-constant map $f:C\to S'$ is that the image of $f$ lies in the primitive linear system of $S'$. Therefore, Corollary \ref{cor:BN_toric_surface} asserts that any $f:C\to S$ mapping generically to the interior of $S$ moves in a family of expected dimension, as long as it is non-degenerate after post-composition with any toric morphism $\pi$ allowed by (i) and (ii) of Theorem \ref{thm:toric_low_degree}.

We now set up the proof of Theorem \ref{thm:toric_low_degree}, which uses standard tools of toric geometry; we follow the conventions of \cite{fulton_toric}. Let $N\cong\mathbb{Z}^2$ be a lattice, and let $S$ be a smooth, projective toric surface with fan $\Sigma\subset N_{\mathbb{R}}$. Let $n_1,\ldots,n_c\in N$ be primitive vectors spanning the 1-dimensional rays of $\Sigma$, ordered counter-clockwise so that the 2-dimensional cone $\sigma_i$ is spanned by $n_i$ and $n_{i+1}$ (indices taken modulo $c$). Let $T \subseteq S$ be the open torus and let $C' \subset S$ be an integral curve generically contained in $T$. 

Let $M=\text{Hom}(N,\mathbb{Z})$ be the dual lattice, and let $F\in \mathbb{C}[M]$ be the equation of $C'$ in $T$. The element $F$ is only well-defined up to multiplication by a unit in $\mathbb{C}[M]$. For a choice of such an $F$, write
   \begin{equation*}
       F=\sum_{m\in M}a_m\chi^m,
   \end{equation*}
   where all but finitely many coefficients $a_m$ are zero. 
   
   \begin{definition}
   Let
   \begin{equation*}
       \cP(F)^+=\{m\in M|a_m\neq 0\},
   \end{equation*}
   and let $\cP(F)$ be the \emph{Newton polygon} of $F$, by definition the convex hull of $\cP(F)^+$. The set $\cP(F)$ and its convex hull are well-defined up to translation in $M$.
   \end{definition}

   For any $i=1,2,\ldots,c$, let $\ell_i\subset M$ be the unique line with the following two properties. First, the function $\ev_i$ on $M=\text{Hom}(N,\mathbb{Z})$ given by evaluation at a non-zero vector on the ray $n_i$ is constant on $\ell_i$. Second, the line $\ell_i$ contains at least one point of $\cP(F)^+$ on which $\ev_i$ is minimized. In particular, every point of $\cP(F)^+$ not lying on $\ell_i$ lies on the same side of $\ell_i$, and every point of $\cP(F)^+$ lying on $\ell_i$ in fact lies on the boundary of $\cP(F)$.
   
   Because $S$ is smooth, the lines $\ell_i,\ell_{i+1}$ necessarily intersect at a lattice point of $M$, which we denote by $\mu_i$. Let $e_i\subset \ell_i$ be the (possibly degenerate) directed segment from $\mu_{i-1}$ to $\mu_i$. Then, the points $\mu_1,\ldots,\mu_c$, some of which may be equal to each other, form the vertices of a convex polygon $\overline{\cP}(F)$ circumscribing $\cP(F)^+$, and appear in order around its boundary. If there are only two distinct points among the $\mu_i$, then the polygon $\overline{\cP}(F)$ is degenerate; in this case, its boundary contains two opposite directed edges. By construction, each edge $e_i$ contains at least one point of $\cP(F)^+$, and all points of $e_i\cap \cP(F)^+$ lie on the boundary of $\cP(F)$, and between the points $\mu_{i-1},\mu_i$. Let $\nu^{-}_{i},\nu^{+}_{i}$ be the points of $e_i\cap \cP(F)^+$ closest to $\mu_{i-1},\mu_i$, respectively.

\begin{example}
    Take $N= \Z^2$ and $S= \P^2$, whose fan has rays in the directions $n_1=(1,0),n_2=(0,1)$ and $n_3=(-1,-1)$. Let $F(x,y)=1+x+y+xy$, where $x=\chi^{e_1^{*}}$ and $y=\chi^{e_2^{*}}$. The polygons $\cP(F)$ and $\overline{\cP}(F)$ are depicted in Figure \ref{fig: example 1}.

    \begin{figure}[H]
    \begin{center}
    \begin{tikzpicture}[xscale=0.55,yscale=0.55]
    %\draw [help lines] (0,0) grid (10, 10);

    %bar{P}(F)
    \node at (0,0) {$\bullet$};

    \draw [ultra thick] (0,0) to (4,0);
    \draw  (-2,0) to (6,0);
    \node at (7,0) {$\ell_2$};
    \node at (4,0.5) {$\mu_2$};
    \node at (2,-0.5) {$\nu_2^+$};
    \node at (2,0) {$\bullet$};
    
    \draw [ultra thick] (4,0) to (0,4);
    \draw  (6,-2) to (-2,6);
    \node at (7,-2.5) {$\ell_3$};
    \node at (0.7,4) {$\mu_3$};
    \node at (3.2,2.5) {$\nu_3^+=\nu_3^-$};
    \node at (2,2) {$\bullet$};
    
    \draw [ultra thick] (0,0) to (0,4);
    \draw  (0,-2) to (0,6);
    \node at (0,7) {$\ell_1$};
    \node at (-2.3,-0.5) {$\nu_1^+=\mu_1= \nu_2^-$};
    \node at (-0.5,2) {$\nu_1^-$};
    \node at (0,2) {$\bullet$};

    %P(F)
    \draw [thick] (0,0) to (0,2);
    \draw [thick] (0,2) to (2,2);
    \draw [thick] (2,0) to (2,2);
    \draw [thick] (0,0) to (0,2);
    
    \end{tikzpicture}
    \caption{The polygons $\cP(F)$ and $\overline{\cP}(F)$ for $F(x,y)=F(x,y)=1+x+y+xy$ in $\P^2$ }\label{fig: example 1}
    \end{center}
    \end{figure}
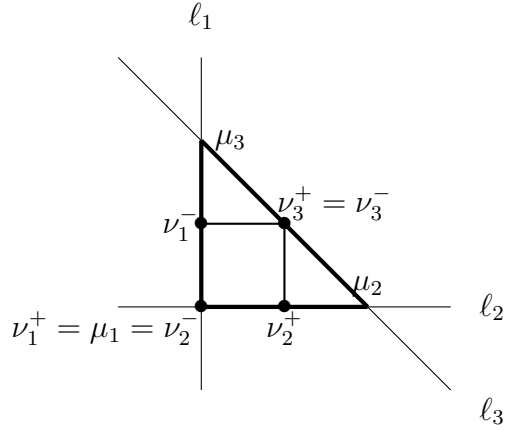
    The Newton polygon $\cP(F)$ has vertices $0,e_1^{*},e_2^{*},e_1^{*}+e_2^{*}\in M$, while $\overline{\cP}(F)$ has vertices $0,2e_1^{*},2e_2^{*}$. Being a conic, the anti-canonical degree of $C'$ in $\P^2$ is equal to $6$, which agrees with Lemma \ref{degree_perimeter} below.
\end{example}

\begin{definition}
       Let $m,m'\in M$ be lattice points. We define the \emph{distance} $\delta(m,m')$ to be equal to 0 if $m=m'$, and equal to one more than the number of lattice points in the interior of the segment between $m,m'$ otherwise.
\end{definition}

   \begin{lemma}\label{degree_perimeter}
     Let $D_i\subset S$ be the toric divisor corresponding to $n_i$, for each $i$. Then, we have $C'\cdot D_i=\delta(\mu_{i-1},\mu_{i})$. In particular, we have
     \begin{equation*}
     C'\cdot (-K_S)=\sum_{i=1}^c C'\cdot D_i=\sum_{i=1}^c \delta(\mu_{i-1},\mu_{i}).
     \end{equation*}
     
   \end{lemma}

   \begin{proof}
       Let $S_i=\text{Spec}(\mathbb{C}[\chi^{m_i},\chi^{m_{i+1}}])$ be the affine chart of $S$ corresponding to $\sigma_i$, where $m_i,m_{i+1}\in M$ form a basis of the dual cone $\sigma_i^{\vee}$ dual to $n_i,n_{i+1}$. The restriction of $C'$ to $S_i$ is cut out by the equation $F_i:=F\cdot \chi^{-\mu_i}\in \mathbb{C}[\chi^{m_i},\chi^{m_{i}+1}]$. Indeed, the monomial $\chi^{-\mu_i}$ is the unique one whose product with $F$ is an irreducible element of $\mathbb{C}[\chi^{m_i},\chi^{m_{i}+1}]$. Multiplication by $\chi^{-\mu_i}$ has the effect of translating the Newton polygon of $\cP(F)^+$ by $-\mu_i$. Moreover, every point $m\in\cP(F_i)^+$ has $m(n_i),m(n_{i+1})\ge 0$, where each inequality is an equality for at least one point $m\in \cP(F_i)^+$ (possibly different for the two inequalities).
       
       The divisor $D_i$ is cut out in $S_i$ by the equation $\chi^{m_{i}}=0$. Setting $\chi^{m_i}=0$ in $F_i$ leaves a polynomial in $\chi^{m_{i+1}}$ of degree equal to $\delta(\nu^{-}_{i},\mu_i)$. Indeed, the non-zero term of $F_i$ given by the largest power of $\chi^{m_{i+1}}$ corresponds (after translation back by $\mu_i$) to the point of $\cP(F)^+$ on $\ell_i$ furthest from $\mu_i$, and hence closest to $\mu_{i-1}$. Thus, the degree of the intersection $C'\cap D_i$ inside $S_i$ is equal to $\delta(\nu^{-}_{i},\mu_i)$. Similarly, the degree of $C'\cap D_i$ upon restriction to the affine chart $S_{i-1}$ is equal to $\delta(\mu_{i-1},\nu^{+}_{i})$. Moreover, the intersection $C'\cap D_i$ has degree $\delta(\nu^{-}_{i},\nu^{+}_{i})$ on $S_{i-1}\cap S_i$ . Therefore, we have
       \begin{equation*}
           C'\cdot D_i=\delta(\nu^{-}_{i},\mu_i)+\delta(\mu_{i-1},\nu^{+}_{i})-\delta(\nu^{-}_{i},\nu^{+}_{i})=\delta(\mu_{i-1},\mu_{i}).
       \end{equation*}
   \end{proof}

   \begin{proof}[Proof of Theorem \ref{thm:toric_low_degree}]
    Let $C'\subset S$ be an integral curve of anti-canonical degree at most 3, which is not contained in the boundary of $S$. By Lemma \ref{degree_perimeter}, the anti-canonical degree of $C'$ is bounded below by the number of (directed) edges of $\overline{\cP}(F)$, which is at least 2. There are two cases.

    First, if $[C']\cdot (-K_S)=2$, then $\overline{\cP}(F)$ consists of two opposite directed edges $e_i,e_j$ of length 1, corresponding to opposite rays $n_i,n_j\in N$. Then, there exists a surjective toric morphism $\pi:S\to \P^1$, corresponding to the projection of $\Sigma$ onto the fan of $\P^1$ that contracts $n_i,n_j$. The curve $C'\subset X_{\Sigma}$ is contracted to a point under $\pi$, so we are in case (i) of Theorem \ref{thm:toric_low_degree}.

    If instead $[C']\cdot (-K_S)=3$, then $\overline{\cP}(F)$ is a triangle consisting of three edges $e_i,e_j,e_k\subset M$ of length 1. These edges correspond to rays $n_i,n_j,n_k\in N$ with $n_i+n_j+n_k=0$. Let $\Sigma'\subset \Sigma$ be the complete fan containing only the three rays $n_i,n_j,n_k$. The corresponding toric variety $S':=Y_{\Sigma'}$ is a fake toric projective plane, and the morphism $S\to S'$ induced by deleting all rays except $n_i,n_j,n_k$ sends $C'$ to a curve in the primitive linear system of $S'$, so we are in case (ii) of Theorem \ref{thm:toric_low_degree}.
   \end{proof}

% We return to setting of a toric surface $S$ from before. Under the toric morphism $b:S\to S'$ corresponding to the contraction of all rays in $\Sigma$ other than $n_i,n_j,n_k$, the curve $C'$ maps birationally to a curve $b(C')\subset S'$ in the primitive linear system. We summarize the situation as follows.

% \begin{theorem}\label{thm:BN_surface}
% Let $C$ be a general curve, and let $S$ be a smooth, projective surface. Let $f:C\to S$ be a map in class $\beta\neq0$, whose image is not contained in the toric boundary of $S$. Suppose that:
% \begin{enumerate}
% \item For every dominant toric morphism $b:S\to\P^1$, the composition $b\circ f$ is not constant, and
% \item For every toric morphism $b:S\to S'$, where $S'$ is a fake toric $\P^2$ arising from deleting rays of $S$, the image of $f$ is not a rational curve in the primitive linear system of $S'$. 
% \end{enumerate}
% Then, the moduli space of maps $\cM_\beta(C,S)$ is generically smooth of expected dimension in a neighborhood of $f$.
% \end{theorem}

% In the first case, the target $\P^1$ can be regarded as the only fake toric $\P^1$, and the condition that $b\circ f$ not be constant can be regarded as the requirement that the image of $C$ not lie in a divisor in the primitive linear system.

% \begin{proof}
%     The above discussion shows that the image of $f$ has anti-canonical degree at least 4, or has anti-canonical degree at least 3 and has positive geometric genus. The Theorem follows from Proposition \ref{prop:multiplecovers} and Corollary \ref{cor:Stransversality}.
% \end{proof}

\subsection{Examples}

Fix again a general curve $C\in\cM_g$. 

\begin{example}
As discussed above, when $S=\bP^2$, Corollary \ref{cor:BN_toric_surface} recovers the classical Brill-Noether Theorem, Theorem \ref{thm:BN_classical}, for $r=2$. 

Indeed, the only curves on $\bP^2$ of anti-canonical degree at most 3 are lines. Correspondingly, the only toric morphism $\pi$ out of $\bP^2$ allowed by Theorem \ref{thm:toric_low_degree} is the identity. Thus, Corollary \ref{cor:BN_toric_surface} asserts that non-constant maps to $\bP^2$ that are not multiple covers of lines move in a family of expected dimension. This is Theorem \ref{thm:BN_classical} in the case $r=2$.
\end{example}

\begin{example}
Let $S=\bP^1\times \bP^1$. A curve $C'$ of bi-degree $(a,b)$ has anti-canonical degree $2(a+b)$, which is at least 4 unless $C'$ is a ruling of $S$, that is, $(a,b)=(1,0)$ or $(a,b)=(0,1)$. The only toric morphisms $\pi$ out of $S$ allowed by Theorem \ref{thm:toric_low_degree} are the projections $\pi:S\to \bP^1$.  Thus, Corollary \ref{cor:BN_toric_surface} asserts that non-constant maps to $\bP^1\times \bP^1$ that are not multiple covers of rulings move in a family of expected dimension. This also follows from the case $r=1$ of Theorem \ref{thm:BN_classical}.
\end{example}

\begin{example}
Let $S=\Bl_q(\P^2)$ be the Hirzebruch surface given by the blow-up of $\P^2$ at a torus-fixed point. In this case, Corollary \ref{cor:BN_toric_surface} recovers Farkas's result, Theorem \ref{thm:farkas}, in dimension 2.

Indeed, write $[C']=d\H-k\E$ for the class of a curve $C'\subset S$, where $\H,\E$ are the hyperplane class (pulled back from $\bP^2$) and the class of the exceptional divisor, respectively. Then, we have
\begin{equation*}
    [C']\cdot(-K_S)=(d\H-k\E)\cdot(3\H-\E)=3d-k.
\end{equation*}
If $C'$ is irreducible and $[C']\cdot(-K_S)\le 3$, then either $(d,k)=(1,0)$ or $(d,k)=(1,1)$, corresponding to the curves $C'$ which map to lines under the toric blow-up $S\to \bP^2$. In the second case, $C'$ is also contracted under the toric projection $S\to\bP^1$. These two toric morphisms are the only ones allowed by Theorem \ref{thm:toric_low_degree}. Corollary \ref{cor:BN_toric_surface} asserts that non-constant maps $f$ to $S$ that are not multiple covers of lines move in a family of expected dimension. It asserts further that the moduli space of maps is generically smooth in a neighborhood of $f$, which does not seem to follow from Farkas's proof of Theorem \ref{thm:farkas}.
\end{example}

\begin{example}
    Let $S$ be the blow-up of $\P^2$ at two torus-fixed points. Then, the toric morphisms $\pi$ allowed by Theorem \ref{thm:toric_low_degree} are again the two projections $S\to\P^1$ and the blow-up $S\to\P^2$. Once more, the fibers of the projections $S\to \P^1$ also map to lines under the blow-up. Thus, any $f:C\to S$ whose image is not contained in a line after the blow-down to $\P^2$ moves in a family of expected dimension. 
    
    One can check that the same is true if $S$ is obtained instead by blowing up $\Bl_q(\P^2)$ at a torus-fixed point of the exceptional divisor.
    
\end{example}

The next example shows that curves which do not map to primitive curves under one contraction may map to primitive curves under others.

\begin{example}
 Let $\pi:S\to\P^2$ be the blow-up of $\P^2$ at its three torus-fixed points. Then, the class $\beta'=2\H-\E_1-\E_2-\E_3$ of a conic $C'$ through the three torus-fixed points has anti-canonical degree 3. By the calculation of Proposition \ref{prop:multiple covers}, multiple covers of such conics $C'$ may fail to move in a family of expected dimension in the moduli spaces of maps $\cM_{m\beta'}(C,S)$, despite the fact that $\pi(C')$ is not a line. However, under the blow-down $\pi':S\to\P^2$ at the other three ($-1$)-curves on $S$, the conics $C'$ contract to lines, that is, they are Cremona-equivalent to lines. The contractions $\pi,\pi'$ are depicted on the level of fans in Figure \ref{fig:cremona}.
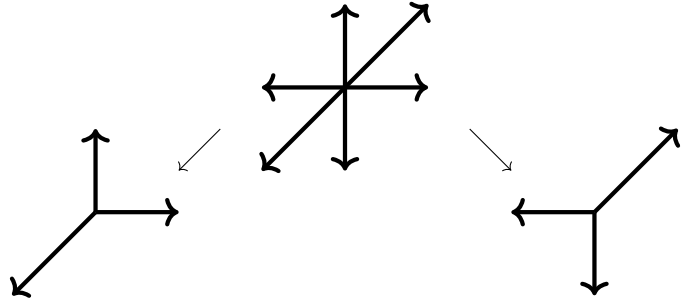
\begin{figure}[H]
    \begin{center}
    \begin{tikzpicture}[xscale=0.55,yscale=0.55]
    %\draw [help lines] (-10,-10) grid (10, 10);

    % --- Blow-up fan in the center ---
    \draw [ultra thick,->] (0,0) -- (2,0);
    \draw [ultra thick,->] (0,0) -- (0,2);
    \draw [ultra thick,->] (0,0) -- (-2,-2);
    \draw [ultra thick,->] (0,0) -- (-2,0);
    \draw [ultra thick,->] (0,0) -- (0,-2);
    \draw [ultra thick,->] (0,0) -- (2,2);

    % --- First fan of P^2 (on the left, lower) ---
    \draw [ultra thick,->] (-6,-3) -- (-4,-3);  % (1,0)
    \draw [ultra thick,->] (-6,-3) -- (-6,-1);  % (0,1)
    \draw [ultra thick,->] (-6,-3) -- (-8,-5);  % (-1,-1)

    % --- Second fan of P^2 (on the right, lower) ---
    \draw [ultra thick,->] (6,-3) -- (8,-1);   % (1,1)
    \draw [ultra thick,->] (6,-3) -- (6,-5);  % (0,-1)
    \draw [ultra thick,->] (6,-3) -- (4,-3);  % (-1,0)

     % --- Arrow to first P2 ---

     \draw [->] (-3,-1) -- (-4,-2);
     \draw [->] (3,-1) -- (4,-2);

    \end{tikzpicture}
    \end{center}
    \caption{Fan of the blow-up of $\mathbb{P}^2$ at three points, with the two contractions to the fan of $\mathbb{P}^2$}\label{fig:cremona}
\end{figure}

\end{example}

The final example illustrates the necessity of considering toric morphisms $\pi:S\to S'$ to \emph{fake} projective planes, as opposed to just projective planes, to witness degeneracy of maps.

\begin{example}\label{example: fake plane}
Let $N=\mathbb{Z}^2$, and let $\Sigma'$ be the complete fan with rays $n_1=(-2,1),n_2=(1,-2),n_3=(1,1)$. Let $S'=Y_{\Sigma'}$ be the corresponding fake $\P^2$, which is a singular toric surface isomorphic to $(\P^2)/(\mathbb{Z}/3)$.
\begin{figure}[H]
    \centering
    \begin{minipage}{0.45\textwidth}
        \centering
        \begin{tikzpicture}[scale=1.5]
            % --- Fan for the toric surface ---
            % Rays of the fake P^2
            \draw[ultra thick,->] (0,0) -- (-2,1);  % n1 = (-2,1)
            \draw[ultra thick,->] (0,0) -- (1,-2);  % n2 = (1,-2)
            \draw[ultra thick,->] (0,0) -- (1,1);   % n3 = (1,1)

            % --- Additional rays in blue ---
            \draw[blue, thick,->] (0,0) -- (1,0);   % (1,0)
            \draw[blue, thick,->] (0,0) -- (0,1);   % (0,1)
            \draw[blue, thick,->] (0,0) -- (-1,1);  % (-1,1)
            \draw[blue, thick,->] (0,0) -- (-1,0);  % (-1,0)
            \draw[blue, thick,->] (0,0) -- (0,-1);  % (0,-1)
            \draw[blue, thick,->] (0,0) -- (1,-1);  % (1,-1)

            % --- Labeling Rays ---
            \node at (-2.3,1) {$n_1$};
            \node at (1.3,-2) {$n_2$};
            \node at (1.2,1.2) {$n_3$};

            % --- Blue Rays Labels ---
            \node[blue] at (1.3,0.1) {$(1,0)$};
            \node[blue] at (0.1,1.2) {$(0,1)$};
            \node[blue] at (-1.2,1.2) {$(-1,1)$};
            \node[blue] at (-1.5,-0.1) {$(-1,0)$};
            \node[blue] at (-0.1,-1.2) {$(0,-1)$};
            \node[blue] at (1.2,-1.2) {$(1,-1)$};
        \end{tikzpicture}
        \caption{Fan $\Sigma'$ in black, and the six added rays to obtain the fan $\Sigma$ in blue}\label{fig:fan_9rays}
    \end{minipage} \hfill
    \begin{minipage}{0.45\textwidth}
        \centering
        \begin{tikzpicture}[scale=1.5]
            % --- Newton Polygon ---
            % Exponents of the terms: (x^2y, xy^2, -3xy, 1)
            % Points: (2,1), (1,2), (1,1), (0,0)
            \draw[ultra thick] (2,1) -- (1,2)  -- (0,0) -- cycle; % Newton polygon

            % --- Additional edge connecting (0,0) and (1,2)

            % --- Newton Polygon Labels ---
            \node at (2.5,1) {$(2,1)$};
            \node at (2,1) {$\bullet$};
            \node at (1.1,2.3) {$(1,2)$};
            \node at (1,2) {$\bullet$};
            \node at (0.8,0.8) {$(1,1)$};
            \node at (1,1) {$\bullet$};
            \node at (0,0) {$\bullet$};
            \node at (-0.3,-0.3) {$(0,0)$};
        \end{tikzpicture}
        \caption{Newton Polygon for \( F(x,y) = x^2y + xy^2 - 3xy + 1 \)}\label{fig:fake_newton_polygon}
    \end{minipage}
\end{figure}

Write $x=\chi^{e_1^*}$ and $y=\chi^{e_2^*}$, and let $C'$ be the curve on the singular fake projective plane $S'=Y_{\Sigma'}$ cut out by the equation $F(x,y)=x^2y+xy^2-3xy+1$. The fan $\Sigma'$ is shown in Figure \ref{fig:fan_9rays}, and $\mathcal{P}(F)=\overline{\mathcal{P}}(F)$ is shown in Figure \ref{fig:fake_newton_polygon}. The curve $C'$ is irreducible, and has arithmetic genus 1 (e.g. by \cite[p. 91]{fulton_toric}) and lies in the primitive linear system on $S'$. On the other hand, the curve $C'$ is singular at the point $(x,y)=(1,1)$, so is rational.

Let $\Sigma$ be the fan obtained by adding the six rays 
\begin{equation*}
    (1,0),(0,1),(-1,1),(-1,0),(0,-1),(1,-1)
\end{equation*} 
to $\Sigma'$, and let $S$ be the smooth toric surface associated to $\Sigma$. Let $\pi:S\to S'$ be the fake projection induced by deletion of rays. Then, the pre-image (or strict transform) $\widetilde{C'}:=\pi^{-1}(C')$ is an irreducible rational curve of anti-canonical degree 3 on $S$, so multiple covers of $\widetilde{C'}$ by general curves may fail to move in a family of expected dimension. On the other hand, because $\cP(F)$ is a triangle containing the point $(1,1)\in M$ in its interior, the curve $C'$ cannot map to a line under any toric birational morphism $S\to \P^2$ induced by deleting rays.

The same assertions hold when $F$ is replaced by the equation of a smooth cubic, such as $F(x,y)=x^2y+xy^2-xy+1$, except that it is no longer true that a general curve (of genus greater than 1) admits multiple covers of $\widetilde{C'}$. 
\end{example}

\bibliographystyle{alpha} 
\bibliography{bib}

\end{document}